\documentclass[12pt,a4paper]{amsart}
\usepackage{enumerate}

\usepackage{amsmath,amssymb,xspace,amsthm}
\newtheorem{theorem}{Theorem}
\newtheorem{remark}[theorem]{Remark}
\newtheorem{lemma}[theorem]{Lemma}
\numberwithin{equation}{section}

\renewcommand{\phi}{\varphi}

\renewcommand{\leq}{\leqslant}
\renewcommand{\geq}{\geqslant}

\begin{document}
\title{Characterization of simple highest weight modules}
\author{Volodymyr Mazorchuk and Kaiming Zhao}
\date{\today}

\begin{abstract}
We prove that for simple complex finite dimensional
Lie algebras, affine Kac-Moody Lie algebras, the
Virasoro algebra and the Heisenberg-Virasoro algebra,
simple highest weight modules are characterized
by the property that all positive root elements
act on these modules locally nilpotently. We
also show that this is not the case for higher rank
Virasoro and for Heisenberg algebras.
\end{abstract}
\maketitle


\vskip 5pt
\noindent
{\em 2000  Math. Subj. Class.:}
17B10, 17B20, 17B65, 17B66, 17B68

\vskip 10pt

\section{Introduction}\label{s0}

Trying to classify all modules over some algebra $A$,
one often faces the following recognition problem: given
some (simple) module $M$ one has to determine whether
$M$ belongs to the class of already known modules.
A common situation is when the module $M$ is not given
explicitly but rather by some general construction 
which allows one to derive easily some rough properties of
$M$ but does not really allow to see subtle properties
of specific elements. It is therefore useful to 
have simple general characterizations for known classes
of $A$-modules. 

If $A$ is the universal enveloping algebra of a Lie
algebra with triangular decompositions, then one of the
most classical families of $A$-modules is formed by the
so-called highest weight modules, see e.g. \cite{MP}
for details and examples. The aim of the present note 
is to prove the following main result which characterizes
simple highest weight modules over certain Lie algebras
with triangular decomposition.

\begin{theorem}\label{thmmain}
Let $\mathfrak{g}$ be one of the following complex Lie
algebras with a fixed triangular decomposition
$\mathfrak{g}=\mathfrak{n}_-\oplus \mathfrak{h}\oplus
\mathfrak{n}_+$ in the sense of \cite{MP}:
\begin{enumerate}[$($a$)$]
\item\label{thmmain.1} a semi-simple finite dimensional
Lie algebra;
\item\label{thmmain.2} an affine Kac-Moody Lie algebra;
\item\label{thmmain.3} the Virasoro Lie algebra; 
\item\label{thmmain.4} the Heisenberg-Virasoro
Lie algebra.
\end{enumerate}
Let $V$ be a $\mathfrak{g}$-module 
(not necessarily weight) on which every 
root element of the algebra $\mathfrak{n}_+$ acts
locally nilpotently. Then  we have the following:
\begin{enumerate}[$($i$)$]
\item\label{thmmain.5} The module $V$ contains a nonzero vector $v$ such that 
$\mathfrak{n}_+\, v=0$.
\item\label{thmmain.6} If $V$ is simple, then $V$ is a highest weight module.
\end{enumerate}
\end{theorem}

Note that the fact that every root element of the 
algebra $\mathfrak{n}_+$ acts locally nilpotently 
on every highest weight module is obvious. The  
proof of Theorem~\ref{thmmain}\eqref{thmmain.5}
is essentially combinatorial and is given in Section~\ref{s1}.
The proof of Theorem~\ref{thmmain}\eqref{thmmain.6} relies on  
irreducibility of generic Verma modules and occupies Section~\ref{s2}.
In Section~\ref{s3} we show that for higher rank Virasoro
and Heisenberg Lie algebras the claim of 
Theorem~\ref{thmmain} is not true and discuss impossibility
of certain natural relaxations of the conditions of 
the Theorem in general.

\section{Proof of Theorem~\ref{thmmain}\eqref{thmmain.5}}\label{s1}

We denote by $\mathbb{N}$ the set of positive integers.
We prove Theorem~\ref{thmmain}\eqref{thmmain.5} using a case-by-case 
analysis of the cases \eqref{thmmain.1}--\eqref{thmmain.4}. 

\subsection{Case of finite dimensional 
Lie algebras}\label{s1.1}

In this subsection, we assume that $\mathfrak{g}$
is as in Theorem~\ref{thmmain}\eqref{thmmain.1}.
In this case $\mathfrak{n}_+$ is finite dimensional
and nilpotent. Hence there is a filtration
\begin{displaymath}
\mathfrak{n}_+= \mathfrak{n}_0\supset
\mathfrak{n}_1\supset\dots\supset
\mathfrak{n}_{\dim \mathfrak{n}_+}=0
\end{displaymath}
such that each $\mathfrak{n}_i$ is an ideal of 
$\mathfrak{n}_{i-1}$ of codimension $1$, stable
with respect to the adjoint action of $\mathfrak{h}$. 
We will use the backward induction on $i$ to show that 
$V$ contains a nonzero element annihilated by 
$\mathfrak{n}_i$ for all $i=0,1,\dots,
\dim \mathfrak{n}_+$. Note that 
$\mathfrak{n}_{\dim \mathfrak{n}_+}$ is 
one-dimensional and generated by a root vector
by our assumption on the filtration. Hence 
$\mathfrak{n}_{\dim \mathfrak{n}_+}$ acts locally 
nilpotently on $V$ giving us the basis of our induction.

To prove the induction step, let $i>0$ and $v\in V$
be such that $v\neq 0$ and $\mathfrak{n}_i\, v=0$.
Let $X\in\mathfrak{n}_{i-1}\setminus \mathfrak{n}_i$ 
be a root element. Then, by our assumptions, 
there exists a nonnegative integer $m\in\mathbb{N}$
such that $w=X^m\, v\neq 0$  and $X^{m+1}\, v=0$.

We claim that $\mathfrak{n}_i w=0$, in fact, we will
prove that $\mathfrak{n}_i X^j\, v=0$ for all 
$j=0,1,\dots,m$ by induction on $j$. The basis  $j=0$ 
of this induction follows from our assumptions. For the
induction step for every $Y\in \mathfrak{n}_i$ we compute:
\begin{displaymath}
Y\cdot X^{j+1}\, v=
Y\cdot X\cdot X^{j}\, v=
X\cdot Y\cdot X^{j}\, v+[Y,X]\cdot X^{j}\, v.
\end{displaymath}
In the latter expression, the first term is zero
directly by induction, and the second term is zero
by induction since $[Y,X]\in \mathfrak{n}_i$ (as
$\mathfrak{n}_i$ is an ideal of $\mathfrak{n}_{i+1}$).
The claim follows.

We note that for the class of weight modules this problem
was considered and solved in \cite{Fe}.

\subsection{Case of affine Kac-Moody 
Lie algebras}\label{s1.2}

In this subsection, $\mathfrak{g}$ is an affine 
Kac-Moody Lie algebra.
Let $\alpha$ be the indivisible positive imaginary root.
Then $\mathfrak{g}_{\alpha}$ is finite dimensional, 
commutative, and acts locally nilpotently on $V$. 
Hence $V$ contains a nonzero $v$ such that 
$\mathfrak{g}_{\alpha}\, v=0$. 

Now let $\beta$ be a positive real root which cannot
be written as a sum of some other positive real root
and a positive imaginary root 
(we call such $\beta$ elementary).
If $\mathfrak{g}_{\beta}\, v=0$, then 
$\mathfrak{g}_{\beta+i\alpha}\, v=0$ for all 
$i\in\mathbb{N}$. If $\mathfrak{g}_{\beta}\, v\neq 0$,
then $\mathfrak{g}_{\alpha}^i\mathfrak{g}_{\beta}\, v=0$
for all $i\gg 0$ by our assumptions. Using 
$\mathfrak{g}_{\alpha}\, v=0$, we obtain
\begin{displaymath}
\mathfrak{g}_{\alpha}\mathfrak{g}_{\beta}\, v=
\mathfrak{g}_{\beta}\mathfrak{g}_{\alpha}\, v+
[\mathfrak{g}_{\alpha},\mathfrak{g}_{\beta}]\, v=
\mathfrak{g}_{\beta+\alpha}\, v
\end{displaymath}
and, similarly, 
$\mathfrak{g}_{\alpha}^i\mathfrak{g}_{\beta}\, v=
\mathfrak{g}_{\beta+i\alpha}\, v$. This implies that
$\mathfrak{g}_{\beta+i\alpha}\, v=0$ for all 
$i\gg 0$ in both cases. Since the number of elementary
positive real roots is finite, we have such a statement
simultaneously and uniformly for all of them.

Note that for any elementary root $\beta$ the
root $-\beta+\alpha$ is also elementary.
Then for every elementary $\beta$
and for all $i\gg 0$ all elements in 
$[\mathfrak{g}_{\beta+i\alpha},
\mathfrak{g}_{-\beta+\alpha +(i-1)\alpha}]$ annihilate $v$.
Such elements generate $\mathfrak{g}_{k\alpha}$
for all $k\gg 0$. The sum of all 
$\mathfrak{g}_{k\alpha}$, $k\gg 0$, and all
$\mathfrak{g}_{\beta+i\alpha}$, $i\gg 0$, for
all elementary roots $\beta$, gives an ideal
$\mathfrak{n}$ of $\mathfrak{n}_+$ of finite
codimension. We have $\mathfrak{n}\, v=0$.
The algebra $\mathfrak{n}_+/\mathfrak{n}$ is finite
dimensional and nilpotent. The proof is now 
completed similarly to the case of finite
dimensional $\mathfrak{g}$, see Subsection~\ref{s1.1}.

We note that for the class of weight modules this result
can be easily deduced from \cite{Fu}.

\subsection{Case of the Virasoro algebra}\label{s1.3}

In this subsection, $\mathfrak{g}$ is the Virasoro
algebra with basis $\{e_i:i\in\mathbb{Z}\}\cup\{c\}$,
where $c$ is central and the rest of the Lie brackets 
is given for $i,j\in\mathbb{Z}$ by
\begin{equation}\label{eq1}
[e_i,e_j]=(j-i)e_{i+j}+\delta_{i,-j}\frac{i^3-i}{12}c.
\end{equation}
The algebras $\mathfrak{n}_{\pm}$ are spanned
(over $\mathbb{C}$) by the elements $e_{\pm i}$,
$i\in\mathbb{N}$, respectively.

As $e_1$ acts on $V$ locally nilpotently, there is
a nonzero $v\in V$ such that $e_1\, v=0$. If
$e_2\,v=0$, then $\mathfrak{n}_+\, v=0$ as
$e_1$ and $e_2$ generate $\mathfrak{n}_+$. Assume
$w=e_2\,v\neq 0$. Then 
\begin{displaymath}
e_1\, w= e_1e_2\,v=e_2e_1\, v+[e_1,e_2]\, v=
e_3\, v.
\end{displaymath}
Similarly one shows that, up to a nonzero scalar,
the element $e_1^i\, w$ coincides with $e_{2+i}\, v$.
As $e_1$ acts on $V$ locally nilpotently,
it follows that there exists a positive integer
$n\geq 2$ such that $e_i\, v=0$ for all $i>n$.

We now show that for every positive integer
$m$ there exists a nonzero $u\in V$ such that
$e_i\, u=0$ for all $i>m$ by a backward induction
on $m$. If $m\geq n$, the claim follows from the previous
paragraph. Let us prove the induction step. Assume
that $0\neq u\in V$ is such that $e_i\, v=0$ for all $i>m$.
If $e_m\, u=0$, the induction step is proved.
Otherwise, for the vector $0\neq u':=e_m\, u$ and for
any $i>m$ we have
\begin{displaymath}
e_i\, u'=e_ie_m\, u=e_me_i\, u+[e_i,e_m]\, u=0 
\end{displaymath}
since $e_i\, u=0$ by the inductive assumption and
$[e_i,e_m]\, u=0$ by the inductive assumption as
$[e_i,e_m]$ equals $e_{i+m}$ up to a nonzero scalar
and $i+m>m$. Now the induction step follows from
the fact that the action of $e_m$ on $V$ is locally
nilpotent.

\subsection{Case of the Heisenberg-Virasoro
algebra}\label{s1.4}

In this subsection, $\mathfrak{g}$ is the 
Heisenberg-Virasoro algebra (cf. \cite{ADKP}) with basis 
$\{e_i, z_i:i\in\mathbb{Z}\}\cup\{c_1,c_2,c_3\}$,
where the $c_i$'s are central and the rest of the 
Lie brackets  is given for $i,j\in\mathbb{Z}$ by
\begin{displaymath}
\begin{array}{rcl}
\left[e_i,e_j\right]&=&(j-i)e_{i+j}+
\delta_{i,-j}\frac{j^3-j}{12}c_1;\\
\left[e_i,z_j\right]&=&jz_{i+j}-
\mathbf{i}j^2\delta_{i,-j}c_2;\\
\left[z_i,z_j\right]&=&j\delta_{i,-j}c_3
\end{array}
\end{displaymath}
(here $\mathbf{i}$ is the imaginary unit).
The algebras $\mathfrak{n}_{\pm}$ are spanned
(over $\mathbb{C}$) by the elements $e_{\pm i}$
and $z_{\pm i}$, where $i\in\mathbb{N}$, respectively.

From the already considered case of the Virasoro
algebra we know that there is a nonzero $v\in V$
such that $e_i\, v=0$ for all $i\in\mathbb{N}$.
If $z_1\, v=0$, then $\mathfrak{n}_+\, v=0$
(as $\mathfrak{n}_+$ is generated by $e_1$, $e_2$
and $z_1$) and we are done. 

If $z_1\, v\neq 0$, then $e_1\, v=0$ implies
(similarly to analogous arguments used several times
above) that $e_1^i\cdot z_1\, v$ equals $z_{i+1}\, v$
up to a nonzero scalar. Since $e_1$ acts locally
nilpotently, it follows that there exists $m\in\mathbb{N}$
such that $z_{i}\, v=0$ for $i>m$. For 
$k\in\mathbb{N}$ let $\mathfrak{n}_i$ denote the
linear span of all $e_i$, $i\in\mathbb{N}$, and all
$z_i$, $i\geq k$. Then $\mathfrak{n}_+=\mathfrak{n}_1$
and each $\mathfrak{n}_i$, $i>1$, is an ideal of 
$\mathfrak{n}_{i-1}$ of codimension one. Since we
already know that $\mathfrak{n}_m\, v=0$, the
proof is completed similarly to the case of finite
dimensional $\mathfrak{g}$, see Subsection~\ref{s1.1}.

\section{Proof of Theorem~\ref{thmmain}\eqref{thmmain.6}}\label{s2}

\subsection{The idea of the proof}\label{s2.1}

By Theorem~\ref{thmmain}\eqref{thmmain.5}, the module $V$ contains
a nonzero element $v$ such that $\mathfrak{n}_+\, v=0$.
Hence there is a short exact sequence
\begin{displaymath}
0\to K\to M\overset{\varphi}{\to} V\to 0,
\end{displaymath}
where $M=U(\mathfrak{g})/U(\mathfrak{g})\mathfrak{n}_+$
is the universal Verma module (see \cite{Zh,Ka}). 
We identify $M$ with $U(\mathfrak{n}_-\oplus\mathfrak{h})
\cong U(\mathfrak{n}_-)\otimes_{\mathbb{C}}U(\mathfrak{h})$
as a $U(\mathfrak{n}_-\oplus\mathfrak{h})$-module
and denote by $\mathbf{v}$ the canonical generator
of $M$. This identification equips $M$ with the
structure of a right $U(\mathfrak{h})$-module,
which commutes with the left $U(\mathfrak{g})$-module
structure by the universal property of Verma modules.
To prove Theorem~\ref{thmmain} it is enough to show that
$V$ is a weight module. The latter is
equivalent to the claim that the ideal 
$I:=K\cap U(\mathfrak{h})$ of $U(\mathfrak{h})$ is maximal.
This is what we are going to prove in the rest of this section.
Note that the ideal $I$ coincides with the annihilator in
$U(\mathfrak{h})$ of the element $\varphi(\mathbf{v})$.

\subsection{The action of $U(\mathfrak{h})$}\label{s2.2}

Let $\Delta$ be the root system of $\mathfrak{g}$ and
$\mathbb{Z}\Delta$ the additive subgroup of $\mathfrak{h}^*$
spanned by $\Delta$.
Fix a basis in each root subspace of $\mathfrak{g}$ 
and a corresponding PBW basis $\{u_i:i\in P\}$ 
of $U(\mathfrak{n}_-)$ (here $P$ is just an indexing set). 
This PBW basis is a basis of $M$ as a 
right (and as a left) $U(\mathfrak{h})$-module.

For $\lambda\in \mathfrak{h}^*$ denote by
$\Phi_{\lambda}:U(\mathfrak{h})\to U(\mathfrak{h})$
the automorphism given by $\Phi_{\lambda}(h)=
h+\lambda(h)$ for all $h\in \mathfrak{h}$.
We have $\Phi_{\lambda}^{-1}=\Phi_{-\lambda}$.
A monomial $u_i$ is said to have weight $\lambda\in \mathbb{Z}\Delta$ 
provided that  $fu_i=u_i\Phi_{\lambda}(f)$ for all 
$f\in U(\mathfrak{h})$.
For $\lambda\in \mathbb{Z}\Delta$ we denote by $P_{\lambda}$
the (finite) set of all indexes $i\in P$ for which
$u_i$ has weight $\lambda$. 

For a $\mathfrak{g}$-module $N$ and an ideal
$J\subset U(\mathfrak{h})$ set 
\begin{displaymath}
N_J:=\{v\in N:J\, v=0\}.
\end{displaymath}
Then, for any $\alpha\in\Delta$ and any nonzero root element
$X_{\alpha}$ we have 
\begin{equation}\label{eq55}
X_{\alpha} N_J\subset N_{\Phi_{\alpha}^{-1}\,J}.
\end{equation}
Denote by $\hat{J}$ the set of all ideals in $U(\mathfrak{h})$
of the form $\Phi_{\lambda}\, J$, $\lambda\in \mathbb{Z}\Delta$.

\begin{lemma}\label{lemn1}
Let $N$ be a $\mathfrak{g}$-module. 
\begin{enumerate}[$($a$)$]
\item\label{lemn1.1} The set 
$\{\mathrm{Ann}_{U(\mathfrak{h})}(v):v\in N,v\neq 0\}$ 
contains an element $J$, maximal with respect to inclusions.
\item\label{lemn1.2} The ideal $J$ above is prime.
\item\label{lemn1.3} If $N$ is simple,
then $N\cong \oplus_{J'\in \hat{J}}N_{J'}$.
\end{enumerate} 
\end{lemma}

\begin{proof}
Claim \eqref{lemn1.1} follows from the fact that $U(\mathfrak{h})$ 
is noetherian. Let $v\in N$ be such that 
$\mathrm{Ann}_{U(\mathfrak{h})}(v)=J$. Assume that $J$ is not prime 
and let $x,y\in U(\mathfrak{h})\setminus J$ be such that $xy\in J$.
Then $w:=y\, v$ is nonzero as $y\not\in J$, moreover, $\mathrm{Ann}_{U(\mathfrak{h})}(w)\supset J$ as $U(\mathfrak{h})$ is 
commutative, and further $\mathrm{Ann}_{U(\mathfrak{h})}(w)$ also 
contains $x\not\in J$. This contradicts maximality of $J$, which 
implies claim \eqref{lemn1.2}. 

If $N$ is simple, it is generated by any nonzero element.
Therefore, as all $\Phi_{\lambda}$, $\lambda\in \mathbb{Z}\Delta$, are 
automorphisms, it follows that all corresponding ideals $\Phi_{\lambda}\, J$
are maximal in $\{\mathrm{Ann}_{U(\mathfrak{h})}(v):v\in N\}$.
Now claim \eqref{lemn1.3} follows from \eqref{eq55}.
\end{proof}

\subsection{The action of the center}\label{s2.3}

For $u=\sum_i u_if_i\in M$ the set of all 
$\lambda\in \mathfrak{h}^*$ for which there exists 
$i\in P_{\lambda}$ such that $f_i\neq 0$ is called
the {\em support} of $u$ and denoted $\mathrm{supp}(u)$.
The element $u$ is called {\em homogeneous} 
if  $|\mathrm{supp}(u)|\leq 1$. 

Recall the
following classical result (see for example 
\cite[Proof of Proposition~3.1]{Mc}, 
or \cite[Theorem~4.2.1(i)]{Ma2} for full details,
see also \cite[2.6.5]{Di} for an alternative argument
in the case of finite dimensional Lie algebras):

\begin{lemma}\label{lemm}
Any endomorphism of a simple module over a countably 
generated associative $\mathbb{C}$-algebra is scalar.
\end{lemma}

Denote by $\mathfrak{z}\subset \mathfrak{h}$ the center
of $\mathfrak{g}$. Definition of $\mathfrak{g}$ equips
$\mathfrak{z}$ with a standard basis (see \cite{MP,ADKP}).
As multiplication with central elements always define
endomorphisms of a module, from Lemma~\ref{lemm} it
follows that $I\cap U(\mathfrak{z})$ is a maximal
ideal in $U(\mathfrak{z})$.

Note that for any homogeneous element $u\in M$ we have 
\begin{equation}\label{eq3}
U(\mathfrak{h})u=uU(\mathfrak{h}).
\end{equation}
Our aim is to show that the ideal $K$ is generated 
by homogeneous elements. If we could show this, then 
\eqref{eq3} would imply that $K$ is stable under the 
right multiplication with $U(\mathfrak{h})$ and hence 
the latter must induce endomorphisms of $V$. Now 
Lemma~\ref{lemm} would imply that $I$ must be maximal 
and we would be done.

\subsection{Reduction to Verma modules}\label{s2.4}

By a similar argument as in the proof of  
Lemma~\ref{lemn1}\eqref{lemn1.2}  we may assume that $I$ is prime. 
Assume that $K$ is not generated by homogeneous elements and 
let $K_h$ be the submodule of $K$ generated by all
homogeneous elements (note that 
$K_h\supset U(\mathfrak{g})I$). Let 
$u\in K\setminus K_h$ be an element such that 
$\mathrm{supp}(u)$ has the minimal possible cardinality 
(note that $|\mathrm{supp}(u)|>1$). Then
$u=\sum_i u_if_i$ and for every $\lambda\in\mathrm{supp}(u)$
the element $u_{\lambda}:=\sum_{i\in P_{\lambda}} u_if_i\not\in K_h$
by the minimality of $\mathrm{supp}(u)$. Without loss of generality we may
also assume that $f_i\not\in I$ whenever $f_i\neq 0$.

As usual, for $\mu,\nu\in \mathfrak{h}^*$ we write
$\mu\preceq\nu$ if $\nu-\mu$ can be written as a
linear combination of positive roots with nonnegative
integer coefficients. Assume additionally
that $\mathrm{supp}(u)$ contains a $\lambda$ which is
maximal (with respect to $\preceq$) in the set of 
all possible elements of the support for all possible
elements from $K\setminus K_h$ with support of the 
minimal possible cardinality. Let $\mu\neq\lambda$ be another 
element of  $\mathrm{supp}(u)$. By the maximality of $\lambda$,
for every positive root $\alpha$ and for any root
element $X_{\alpha}$ in $\mathfrak{g}_{\alpha}$
we get  $X_{\alpha}u\in K_h$, in particular,
$X_{\alpha}u_{\lambda}\in K_h$ for every $\lambda\in\mathrm{supp}(u)$. 

We have $\varphi(\mathbf{v})\in V_{I}$ by our assumption.
The element $u_{\lambda}$ is homogeneous and does not belong to 
$K_h$. Hence it does not belong to $K$ either and thus
$u_{\lambda}\varphi(\mathbf{v})\neq 0$, which implies that
$u_{\lambda}\varphi(\mathbf{v})\in V_{\Phi_{\lambda}^{-1}(I)}$
by \eqref{eq55}.
On the other hand, $u\varphi(\mathbf{v})=0$, which implies that 
\begin{displaymath}
u_{\lambda}\varphi(\mathbf{v})=-\sum_{\nu\in\mathrm{supp}(u)
\setminus\{\lambda\}} u_{\nu}\varphi(\mathbf{v}).
\end{displaymath}
The annihilator in $U(\mathfrak{h})$ of the element on the right hand 
side equals 
\begin{displaymath}
I':=\bigcap_{\nu\in\mathrm{supp}(u)
\setminus\{\lambda\}}\Phi_{\nu}^{-1}(I).
\end{displaymath}
Therefore $\Phi_{\lambda}^{-1}(I)=I'$ is a prime ideal.
Since all $\Phi_{\xi}$ are automorphisms, the ideals 
$\Phi_{\lambda}^{-1}(I)$ and $\Phi_{\nu}^{-1}(I)$ have the same
height and the same depth, which implies 
$\Phi_{\lambda}^{-1}(I)=\Phi_{\nu}^{-1}(I)$
for all $\nu\in\mathrm{supp}(u)
\setminus\{\lambda\}$. In particular, it follows that
the ideal $I$ is invariant under $\Phi_{\lambda-\mu}$
for $\lambda\neq \mu$ as fixed above.

Let $\mathbf{m}$ be any maximal ideal of $U(\mathfrak{h})$
containing $I$. Assume that $\mathbf{m}$ is given by
$\lambda_{\mathbf{m}}\in \mathfrak{h}^*$.
Then $U(\mathfrak{g})\mathbf{m}\supset U(\mathfrak{g})I$.
Consider the Verma module 
$M(\lambda_{\mathbf{m}}):=M/U(\mathfrak{g})\mathbf{m}$
with  highest weight $\lambda_{\mathbf{m}}$.
Then $K_h+ U(\mathfrak{g})\mathbf{m}$ is a submodule
of $M(\lambda_{\mathbf{m}})$ and the intersection of this
submodule with $U(\mathfrak{h})$ equals $\mathbf{m}$.
Hence the corresponding quotient $Q$ is nonzero.
We may even choose $\mathbf{m}$ such that the images of 
the nonzero $f_i$'s in both $u_{\lambda}=\sum_{i\in P_{\lambda}} u_if_i$ and 
$u_{\mu}=\sum_{i\in P_{\mu}} u_if_i$ are nonzero, which yields that the
images of both $u_{\lambda}$ and $u_{\mu}$ in
$Q$ are nonzero. Then
$Q$ contains a nonzero primitive vector of weight 
$\lambda_{\mathbf{m}}+\lambda$ and a nonzero primitive 
vector of weight  $\lambda_{\mathbf{m}}+\mu$.  
By \cite[2.11]{MP}, existence of
a primitive vector  in $Q$ implies existence of a primitive vector
in $M(\lambda_{\mathbf{m}})$ of the same weight.

\subsection{Completion of the proof}\label{s2.5}

As we have seen above, the ideal 
$I$ is invariant under the action of $\Phi_{\lambda-\mu}$.
This implies that the ideal $I$ is generated by its
intersection with $U(\mathfrak{h}')$, where
$\mathfrak{h}'$ consists of all $h\in \mathfrak{h}$
such that $(\lambda-\mu)(h)=0$. Indeed, 
let $\overline{h}$ be a nonzero element from the complement
of $\mathfrak{h}'$ in $\mathfrak{h}$, which can be
written as an integral linear combination of coroots. If
$f\in I$, then $f-\Phi_{\lambda-\mu}(f)\in I$
and the latter has a strictly smaller degree with 
respect to $\overline{h}$.

The last paragraph means that, when choosing $\mathbf{m}$
above, we are free to choose any eigenvalue of
$\overline{h}$ (the only requirement is that
it should not kill the images of the $u_{\lambda}$
and $u_{\mu}$ in $Q$, but this restriction
means that a finite number of nonzero polynomials in
$\overline{h}$ should not vanish). 
In particular, we can choose this eigenvalue to
be a complex number, which is transcendental over 
the finite extension of $\mathbb{Q}$ given by
adjoining eigenvalues of our fixed (finite) basis in 
$\mathfrak{z}$. Then the  usual structure theory of 
Verma modules (see \cite{ADKP,KK,MP}) says that 
reducibility and submodules of Verma modules are 
controlled by vanishing of the Shapovalov form (see 
\cite{Sh,ADKP,KK} or \cite[2.8]{MP}), whose determinant is given 
in  terms of certain polynomials over $\mathfrak{h}$
with rational coefficients. This means that in the
case the eigenvalue of $\overline{h}$ is as chosen above, it is 
not possible  for a Verma module to have nonzero eigenvectors 
of weights $\lambda_{\mathbf{m}}+\lambda$ and 
$\lambda_{\mathbf{m}}+\mu$ at the same time.
The obtained contradiction completes the proof.

\section{Some (counter)examples}\label{s3}

\subsection{Higher rank Virasoro algebras}\label{s3.1}

Let $G\subset \mathbb{R}$ be a nontrivial
additive subgroup, which is not isomorphic to 
$\mathbb{Z}$, and $\mathrm{Vir}_G$ be the
correspondent higher rank Virasoro algebra as 
in \cite{Za,PZ}. It has basis 
$\{e_i:i\in G\}\cup\{c\}$, where $c$ is central 
and the rest of the Lie brackets  is given by \eqref{eq1}.
The subalgebra $\mathfrak{n}_{+}$ and
$\mathfrak{n}_{-}$ are spanned by $e_i$ with
$i>0$ and $i<0$, respectively. Let $M(0)$ be the Verma
module over $\mathrm{Vir}_G$ whose simple top $L(0)$
is the trivial $\mathrm{Vir}_G$-module. Denote by
$K(0)$ the kernel of the canonical 
epimorphism $M(0)\to L(0)$
(see \cite{Ma}). In \cite{HWZ} it was shown that
$K(0)$ is simple. Clearly, every root vector of
$\mathfrak{n}_+$ acts locally nilpotently on $K(0)$.
At the same time $K(0)$ does not have any highest weight.
Indeed, its support coincides with the set of all
negative elements of $G$. As $G\not\cong\mathbb{Z}$,
$G$ contains negative elements of arbitrarily small
absolute value. This means that Theorem~\ref{thmmain}
does not hold for $\mathrm{Vir}_G$.

\subsection{Heisenberg Lie algebra}\label{s3.3}

The Heisenberg Lie algebra $\mathfrak{H}$ has basis
$\{z_i:i\in\mathbb{Z}\}$ and the Lie bracket is given by
\begin{displaymath}
[z_i,z_j]=j\delta_{i,-j}z_0.
\end{displaymath}
The subalgebras $\mathfrak{n}_{+}$ and
$\mathfrak{n}_{-}$ are spanned by $z_i$ with
$i>0$ and $i<0$, respectively.

Consider the set 
\begin{displaymath}
I=\{\varepsilon=(\varepsilon_1,\varepsilon_2,\dots):
\varepsilon_i\in\mathbb{N};
\varepsilon_i=2\text{ for all }i\gg 0\}
\end{displaymath}
and let $V$ have basis $\{v_{\varepsilon}:\varepsilon\in I\}$.
For $i\in\mathbb{Z}$ and $\varepsilon\in I$ set
\begin{equation}\label{eq2}
z_i\, v_{\varepsilon}=
\begin{cases}
v_{\varepsilon}, & i=0;\\
0, & i>0\text{ and }\varepsilon_i=1;\\
v_{(\varepsilon_1,\dots,\varepsilon_{i-1},\varepsilon_i-1,
\varepsilon_{i+1},\varepsilon_{i+2},\dots)}, 
& i>0\text{ and }\varepsilon_i>1;\\
i\varepsilon_{-i}
v_{(\varepsilon_1,\dots,\varepsilon_{i-1},\varepsilon_i+1,
\varepsilon_{i+1},\varepsilon_{i+2},\dots)}, 
& i<0.
\end{cases}
\end{equation}
It is easy to check that this
makes $V$ into an $\mathfrak{H}$-module.
Further $z_i^{\varepsilon_i}\, v_{\varepsilon}=0$, $i>0$,
and hence such $z_i$ acts on $V$ locally nilpotently.

We claim that $V$ is simple. Assume that this is not 
the case and let $W\subset V$ be a proper nonzero 
submodule. Let $k\in\mathbb{N}$ be minimal such that 
$W$ contains a nontrivial linear combination of basis
elements with exactly $k$ nonzero summands. Let
$u\in V$ be such a combination. Note that
$V$ is clearly generated by any $v_{\varepsilon}$
and hence $k>1$. Then there
exists $\varepsilon,\varepsilon'\in I$
and $i\in\mathbb{N}$ such that 
$\varepsilon_i<\varepsilon'_i$ and  both
$v_{\varepsilon}$ and $v_{\varepsilon'}$ appear in
$u$ with nonzero coefficients. This implies that
$z_i^{\varepsilon_i}\, u\in W$, on the one hand, 
is nonzero, but, on the other hand, contains less
then $k$ nonzero summands, a contradiction.

At the same time, every nonzero vector of $V$ generates 
an infinite-dimensional $\mathfrak{n}_+$-submodule 
(applying $z_i$ for $i\gg 0$). Hence Theorem~\ref{thmmain} 
does not hold for $\mathfrak{H}$.

\begin{remark}\label{rem}
{\em 
The construction above admits a straightforward
generalization: for $\mathbf{e}\in\mathbb{N}^\mathbb{N}$
consider the set 
\begin{displaymath}
I_{\mathbf{e}}=\{\varepsilon=
(\varepsilon_1,\varepsilon_2,\dots):
\varepsilon_i\in\mathbb{N};
\varepsilon_i=\mathbf{e}_i\text{ for all }i\gg 0\}
\end{displaymath}
and define the $\mathfrak{H}$-module structure on 
the vector spaces $V_{\mathbf{e}}$ with basis 
$\{v_{\varepsilon}:\varepsilon\in I_{\mathbf{e}}\}$
using
\eqref{eq2}. Similarly to the above one shows that
the module $V_{\mathbf{e}}$ is simple. These modules
generalize simple $\mathfrak{H}$-modules constructed in
\cite[Section~6]{BBF}.
}
\end{remark}

A couple of months after the original version of this paper was
published on the arxive, similar modules appeared in a more
general context in \cite{BBFK}.

\subsection{Relaxing the conditions}\label{s3.4}

One might wonder whether conditions of Theorem~\ref{thmmain}
could be relaxed in some natural way (for example,
by requiring the local nilpotency only for the
generators of $\mathfrak{n}_+$). For instance,
in the case of affine Lie algebras there are two types
of roots, real and imaginary, and the algebra
$\mathfrak{n}_+$ is generated by the real roots.
It is therefore tempting to ask whether it would be
enough in Theorem~\ref{thmmain} to require that 
the action of all real root vectors is locally nilpotent. 
Unfortunately, this would affect the statement as the
adjoint representation gives rise to a simple
non-highest weight module on which all real root vectors
acts locally nilpotently. We do not know which class
of simple modules is described by the conditions relaxed
in this way.
\vspace{1cm}

\begin{center}
\bf Acknowledgments
\end{center}

The research was done during the visit of the first 
author to Wilfrid Laurier University in April 2011. 
The hospitality and financial support of Wilfrid 
Laurier University are gratefully acknowledged. 
The first author was partially supported by the 
Swedish Research Council. The second author was 
partially supported by NSERC. We thank the referee
for helpful comments and suggestions.

\vspace{1cm}

\noindent
V.M.: Department of Mathematics, Uppsala University, 
Box 480, SE-751 06, Uppsala, Sweden; e-mail: {\tt mazor\symbol{64}math.uu.se}
\vspace{0.5cm}

\noindent K.Z.: Department of Mathematics, Wilfrid Laurier
University, Waterloo, Ontario, N2L 3C5, Canada; and College of Mathematics and 
Information Science, Hebei Normal (Teachers) University, Shijiazhuang 050016, 
Hebei, P. R. China. e-mail:  {\tt kzhao\symbol{64}wlu.ca}

\end{document}